\documentclass[10pt]{amsproc}
\usepackage{hyperref,xy,doi}
\xyoption{all}

\newtheorem{theorem}{Theorem}[section]
\newtheorem{corollary}[theorem]{Corollary}
\newtheorem{lemma}[theorem]{Lemma}
\theoremstyle{definition}
\newtheorem{definition}[theorem]{Definition}
\theoremstyle{remark}
\newtheorem{remark}[theorem]{Remark}
\newtheorem{example}[theorem]{Example}
\DeclareMathOperator{\Res}{Res}
\DeclareMathOperator{\Ind}{Ind}
\DeclareMathOperator{\Hom}{Hom}
\DeclareMathOperator{\Rep}{Rep}
\DeclareMathOperator{\Sym}{Sym}
\newcommand{\GL}{\mathrm{GL}}
\newcommand{\CC}{\mathbf C}
\newcommand{\sgn}{\mathrm{sgn}}
\newcommand{\tr}{\mathrm{trace}}
\DeclareMathOperator{\ch}{\mathrm{ch}}
\newcommand{\diag}{\mathrm{diag}}
\newcommand{\NN}{\mathbf N}
\newcommand{\xx}{\mathbf x}

\title{Polynomial Induction and the Restriction Problem}

\author[Narayanan]{Sridhar P. Narayanan}
\address{The Institute of Mathematical Sciences (HBNI), Chennai}
\email{sridharn@imsc.res.in}

\author[Paul]{Digjoy Paul}
\address{The Institute of Mathematical Sciences (HBNI), Chennai}
\email{digjoypaul@gmail.com}

\author[Prasad]{Amritanshu Prasad}
\address{The Institute of Mathematical Sciences (HBNI), Chennai}
\email{amri@imsc.res.in}

\author[Srivastava]{Shraddha Srivastava}
\address{Tata Institute of Fundamental Research, Mumbai}
\email{maths.shraddha@gmail.com}

\begin{document}
\maketitle
\begin{abstract}
  We construct the polynomial induction functor, which is the right adjoint to the restriction functor from the category of polynomial representations of a general linear group to the category of representations of its Weyl group.
  This construction leads to a representation-theoretic proof of Littlewood's plethystic formula for the multiplicity of an irreducible representation of the symmetric group in such a restriction.
  The unimodality of certain bipartite partition functions follows.
\end{abstract}

\section{Introduction}
\label{sec:intro}
A polynomial representation of $\GL_n(\CC)$ on a finite dimensional $\CC$-vector space $W$ is a homomorphism $\rho:\GL_n(\CC)\to\GL(W)$ such that, with respect to any basis of $W$, the matrix entries of $\rho(g)$ are polynomials in the matrix entries of $g\in \GL_n(\CC)$.
The character of such a representation is the symmetric polynomial:
\begin{displaymath}
  \ch\rho(t_1,\dotsc,t_n) = \tr(\rho(\diag(t_1,\dotsc,t_n));W).
\end{displaymath}
Here $\diag(t_1,\dotsc,t_n)$ denotes the diagonal matrix with entries $t_1,\dotsc,t_n$.

The polynomial representation $\rho$ is said to be homogeneous of degree $d$ if the matrix entries of $\rho(g)$ are all homogeneous of degree $d$.
Schur \cite{schurthesis} showed that every polynomial representation of $\GL_n(\CC)$ is isomorphic to a direct sum of irreducible homogeneous polynomial representations.
Moreover, the set of isomorphism classes of irreducible homogeneous polynomial representations of $\GL_n(\CC)$ of degree $d$ is indexed by $\Lambda(d,n)$, the set of partitions of $d$ with at most $n$ parts.
For $\lambda\in \Lambda(d,n)$ we denote by $(\rho_\lambda,W_\lambda)$ the irreducible polynomial representation of $\GL_n(\CC)$ of degree $d$ corresponding to $\lambda$.
Schur proved that $\ch \rho_\lambda(t_1,\dotsc,t_n)$ is the Schur polynomial $s_\lambda(t_1,\dotsc,t_n)$.
For details see \cite[Chapter~6]{rtcv}.

Frobenius \cite{frob} classified the irreducible complex representations of the symmetric groups $S_n$ via their characters.
For every partition $\mu$ of $n$, there exists a unique irreducible representation $V_\lambda$ of $S_n$, called the Specht module associated to $\lambda$.
As $\lambda$ runs over the set $\Lambda(n)$ of all partitions of $n$, the Specht modules $V_\lambda$ form a complete set of representatives of isomorphism classes of irreducible complex representations of $S_n$ (see \cite[Chapter~3]{rtcv}).

Regard the symmetric group $S_n$ as the subgroup of permutation matrices in $\GL_n(\CC)$.
Let $\Rep^d \GL_n(\CC)$ denote the category of homogeneous polynomial representations of $\GL_n(\CC)$ of degree $d$.
Let $\Rep S_n$ denote the category of complex representations of $S_n$.
Let $\Res^d: \Rep^d \GL_n(\CC)\to \Rep S_n$ denote the functor that takes a polynomial representation $\rho:\GL_n(\CC)\to \GL(W)$ to $\rho_{|S_n}:S_n\to \GL(W)$.
Let $r_{\lambda\mu}$, for $\lambda\in \Lambda(d,n)$ and $\mu \in \Lambda(n)$, be non-negative integers defined by
\begin{equation}
  \label{eq:restriction}
  \Res^d W_\lambda = \bigoplus_{\mu\in \Lambda(n)} V_\mu^{\oplus r_{\lambda\mu}}.
\end{equation}
The integers $r_{\lambda\mu}$ are known as \emph{restriction coefficients}.
The best-known way of computing restriction coefficients is by expanding a \emph{plethysm} of symmetric functions in the basis of Schur functions (Theorem~\ref{lit}), which we survey in greater detail in the next two sections.
Orellana and Zabrocki \cite[Theorem~1]{2016arXiv160506672O}, and independently, Assaf and Speyer \cite{assaf-speyer} obtained restrictions coefficients from the expansions of Schur functions into a new family of symmetric functions (the Specht symmetric functions).
A recent approach due to the authors of this paper involves character polynomials \cite{ASSD}.
However, none of these approaches gives a positive \emph{combinatorial interpretation} of the restriction coefficients.
This remains an important open question in algebraic combinatorics, known as the restriction problem.

Frobenius \cite{frob-ind} introduced the notion of an induced character proved the reciprocity theorem named after him.
The Frobenius reciprocity theorem can be interpreted as the adjunction between restriction and induction functors.
Frobenius's ideas were extended to locally compact topological groups and their unitary representations by Mackey~\cite{Mackey}.
The adaptation of Mackey's construction to the setting of polynomial representations is completely natural.
In this article, we construct a polynomial induction functor $\Ind^d:\Rep S_n\to \Rep^d \GL_n(\CC)$ which is a right adjoint to $\Res^d$ (Theorem~\ref{theorem:adj}).
We show that the character of $\Ind^d V_\mu$ is given by a plethysm (Theorem~\ref{theorem:indchar}).
This gives rise to a representation-theoretic proof of Littlewood's plethystic formula for restriction coefficients (Theorem~\ref{lit}).
\section{Plethysm and the Restriction Problem}
\label{sec:plethysm-restr-probl}
\subsection{Plethysm}
\label{sec:plethysm}
Composing a polynomial representation $\phi:GL_n \rightarrow GL_m$ with a polynomial representation $\psi:GL_m \rightarrow GL_r$ results in the polynomial representation $\psi \circ \phi:GL_n \rightarrow GL_r$. Let $f$ and $g$ be the characters of $\psi$ and $\phi$ respectively. The plethysm $f[g]$ is defined to be the character of $\psi \circ \phi$.
The concept of plethysm was introduced by Littlewood in \cite{MR0002127}.

A detailed discussion of the representation-theoretic significance of plethysm can be found in \cite[Section~A2.6]{MR1676282}.
Let $\NN$ denote the set of all non-negative integers.
For $\xx=(x_1,\dotsc,x_n)\in \NN^n$, let $t^\xx$ denote the monomial $t_1^{x_1}\dotsb t_n^{x_n}$.
Since $g$ is the character of the $m$-dimensional representation $\phi$,  $g$ can be written as a sum of $m$ monic monomials:
\begin{displaymath}
  g(t)=t^{\xx_1}+\dotsb+t^{\xx_m},\text{ for some }\xx_1,\dotsc,\xx_m\in \NN^n.
\end{displaymath}
For any symmetric polynomial $f$ in $m$ variables, the plethysm $f[g]$ can be defined as the substitution of these monic monomials of $g$ into the variables of $f$ (see \cite[Theorem 7]{MR2765321}):
\begin{displaymath}
  f[g](t) = f(t^{\xx_1},\dotsc,t^{\xx_m}).
\end{displaymath}
This definition also works when $f$ is a symmetric polynomial in infinitely many variables \cite[Section~7.1]{MR1676282} and $g$ is a sum of infinitely many monic monomials.
\begin{example}
  \label{eg:pl1}
  Let $f(t_1,\dotsc,t_4)$ be any symmetric polynomial and $g=(t_1+t_2)^2$.
  Since $g=t_1^2+t_2^2+t_1t_2+t_1t_2$, $f[g]=f(t_1^2,t_2^2,t_1t_2,t_1t_2)$.
\end{example}
\begin{example}
  \label{eg:pl2}
  For each non-negative integer $k$, let $p_k=\sum_i t_i^k$ be the $k$th power sum symmetric polynomial.
  The plethysm $p_k[p_l]$ is obtained by substituting the monomial $t_i^l$ of $p_l$ for the variable $t_i$ of $p_k$, so that $p_k[p_l]=p_{kl}$.
\end{example}
It is clear from this definition that given a symmetric function $g$, the map $R_g:f\mapsto f[g]$ is an algebra homomorphism. In other words, for symmetric functions $f_1$ and $f_2$, we have:
\begin{itemize}
\item $(f_1+f_2)[g]=f_1[g]+f_2[g]$,
\item $f_1f_2[g]=f_1[g]f_2[g]$.
\end{itemize}
Note that $R_{p_1}$ is the identity homomorphism. 

Finding a combinatorial interpretation of the coefficients $a_{\lambda \mu \nu}$ in the expression $s_\lambda[s_\mu]=\sum_{\nu} a_{\lambda \mu \nu}s_\nu$ remains an open problem in general. Littlewood \cite{MR0002127} obtained generating functions for $s_2[s_n],s_{1^2}[s_n],s_n[s_2], s_n[s_{1^2}]$, while Carr\'e and Leclerc \cite{MR1331743} offered combinatorial interpretations for $s_2[s_\lambda]$ and $s_{1^2}[s_\lambda]$. Other partial results can be found in \cite{CRY,CRY2,MR3685118,MR1661367,MR777698,MR2035305}.
For each non-negative integer $k$, let $h_k$ and $e_k$ denote the $k$th complete and elementary symmetric functions, respectively:
\begin{align}
  \label{eq:hk}
  h_k(t_1,t_2,\dotsc) & = \sum_{i_1\leq \dotsb \leq i_k} t_{i_1}\dotsb t_{i_k}\\
  \label{eq:ek}
  e_k(t_1,t_2,\dotsc) & = \sum_{i_1<\dotsb<i_k} t_{i_1}\dotsb t_{i_k}.
\end{align}
Consider the formal power series
\begin{displaymath}
  H(t) = H(t_1,\dotsc,t_n) = \sum_{k\geq 0} h_k(t_1,\dotsc,t_n).
\end{displaymath}
The plethysm of $H(t)$ into various symmetric functions plays an important role in this paper.
Observe that
\begin{displaymath}
  H(t) = \sum_{\xx\in \NN^n} t^\xx.
\end{displaymath}
\begin{definition}
  \label{definition:vecpar}
  Given $\xx\in \NN^n$, a vector partition of $\xx$ with $k$ parts is a decomposition
  \begin{displaymath}
    \xx = \xx_1+\dotsb+\xx_k,
  \end{displaymath}
  where $\xx_1,\dotsc,\xx_k\in \NN^n-\{\mathbf 0\}$.
  The order in which the summands appear is disregarded.
  The number of vector partition of $\xx$ with at most $k$ parts is denoted $p_k(\xx)$.
  The number of vector partition of $\xx$ with exactly $k$ or $k-1$ \emph{distinct} parts in $\NN^n-\{\mathbf 0\}$ is denoted $q_k(\xx)$.
  These functions can also be defined by the generating functions
  \begin{align*}
    \sum_{\xx\in \NN^n} \sum_{k\geq 0} p_k(\xx)t^\xx u^k & = \prod_{\xx\in\NN^n}\frac 1{1-t^\xx u},\\ 
    \sum_{\xx\in \NN^n} \sum_{k\geq 0} q_k(\xx)t^\xx u^k & = \prod_{\xx\in\NN^n}(1+t^\xx u).\\ 
  \end{align*}
\end{definition}
\begin{theorem}
  \label{theorem:ehH}
  For every non-negative integer $k$,
  \begin{align*}
    h_k[H] &=\sum_{\xx\in \NN^n} p_k(\xx)t^\xx,\\
    e_k[H] &=\sum_{\xx\in \NN^n} q_k(\xx)t^\xx.
  \end{align*}
\end{theorem}
\begin{proof}
  Endow $\NN^n$ with the lexicographic order.
  The plethysm $h_k[H]$ is obtained by substituting the monomials $\{t^\xx\mid \xx\in\NN^n\}$ into the variables of $h_k$ in order.
  Thus by \eqref{eq:hk}, the coefficient of $t^\xx$ in $h_k[H]$ is the number of decompositions
  \begin{displaymath}
    \xx = \xx_1+\dotsb+\xx_k, \text{ where } \xx_1\leq \dotsb\leq \xx_k\in\NN^n,
  \end{displaymath}
  which is the same as $p_k(\xx)$.
  Similarly, by \eqref{eq:ek}, the coefficient of $t^\xx$ in $e_k[H]$ is the number of decompositions
  \begin{displaymath}
    \xx = \xx_1+\dotsb+\xx_k, \text{ where } \xx_1<\dotsb<\xx_k\in\NN^n,
  \end{displaymath}
  namely $q_k(\xx)$.
\end{proof}
\begin{remark}
  By expressing the generating functions of $p_k$ and $q_k$ as plethysms of Schur positive symmetric functions, Theorem~\ref{theorem:ehH} establishes their Schur positivity.
  For a wider discussion of Schur positivity see \cite{Peal}.
\end{remark}
\subsection{Plethystic Formula for Restriction Coefficients}
\label{sec:plethyst-form-restr}
Let $\theta_1,\ldots, \theta_n$ be the eigenvalues of the permutation matrix $w\in S_n\subset GL_n(\CC)$.
Then the character of $\Res^d W_\lambda$ at $w$ is $s_{\lambda}(\theta_1,\ldots, \theta_n)$.
Let $\langle\cdot,\cdot\rangle$ denote the Hall inner product on symmetric functions.
For a representation $(\rho, V)$ of $S_n$, let $\mathcal{F}(V)$ denote the Frobenius characteristic of the character of $V$ \cite[Section~7.18]{MR1676282}.
We have $\mathcal{F}(V_\lambda)=s_{\lambda}$ and $\dim \Hom_{S_n}(V_1,V_2)=\langle \mathcal{F}(V_1) , \mathcal{F}(V_2) \rangle$ for representations $V_1$ and $V_2$ of $S_n$.
By applying $\mathcal F$ to both sides of \ref{eq:restriction}, we get
\begin{displaymath}
  r_{\lambda \mu}=\langle \mathcal{F}(\Res^d W_\lambda), s_{\mu}\rangle.  
\end{displaymath}

The following plethystic formula is equivalent to a result of Littlewood \cite[Theorem~XI]{MR95209}:
\begin{theorem}\label{lit}
  For every $\lambda\in \Lambda(d,n)$ and $\mu \in \Lambda(n)$,
  $$r_{\lambda\mu}=\langle \mathcal{F} (\Res^d W_\lambda), s_{\mu}\rangle = \langle s_{\lambda}, s_{\mu}[H]\rangle.$$
\end{theorem}
Stanley \cite[Exercise 7.74]{MR1676282} outlines a proof using symmetric function theory. For a proof using Hopf algebra techniques see Scharf and Thibon \cite[Corollary 5.3]{MR1272068}, who also explain the equivalence of this identity and Littlewood's branching rule for $GL_{n-1}(\CC)\supset S_n$ \cite[Theorem XI]{MR95209}.

Let $\lambda=(\lambda_1,\dotsc,\lambda_n)$ be a weak composition of $d$, and $\Sym^\lambda \CC^n = \otimes_{i=1}^n \Sym^{\lambda_i}\CC^n$.
It is known from the work of Orellana and Zabrocki \cite[Theorem~9]{2016arXiv160506672O} and Harman \cite[Proposition 3.11]{Nate} that the multiplicity $a_{\lambda\mu}$ of the Specht module $V_\mu$ in $\Res^d \Sym^\lambda \CC^n$ is counted by a certain class of multiset tableaux.
In \cite[Theorem 4.2]{NPS} it is shown that $s_{\mu}[H]$ is the generating function for $\{a_{\lambda \mu}\}$ as $\lambda$ varies. A combinatorial proof of Theorem~\ref{lit} is given in the second author's PhD thesis \cite{dpthesis} using this fact and some symmetric function identities.

In this paper, we show that $\ch \Ind^d V_\mu$ is $s_\mu[H]$ and thereby give a representation-theoretic proof of Theorem~\ref{lit}.

\section{Polynomial Induction}
\label{sec:polynomial-induction}

\subsection{Definition of the Induction Functor}
\label{sec:induction-functor}
Let $M_n$ denote the ring of $n\times n$ matrices with entries in $\CC$.
Let $P^d(M_n)$ denote the space of homogeneous polynomials of degree $d$ in the entries of matrices $Q\in M_n$.
Then, for any complex vector space $V$, $P^d(M_n)\otimes V$ can be regarded as the space of $V$-valued homogeneous polynomials of degree $d$ on $M_n$.
Define a functor $\Ind^d:\Rep S_n\to \Rep^d \GL_n(\CC)$ as follows: given an object $(\rho,V)$ of $\Rep^d S_n$ let
\begin{equation}
  \label{eq:indv}
  \Ind^d V = \{f\in P^d(M_n)\otimes V\mid f(wQ)=\rho(w)f(Q)\text{for all $w\in S_n$, $Q\in M_n$}\}.
\end{equation}
The vector space $\Ind^d V$ affords a representation of $\GL_n(\CC)$ via the action
\begin{equation}
  \label{eq:indrho}
  (\Ind^d \rho(g) f)(Q) = f(Qg)
\end{equation}
for all $f\in \Ind^d V$, $g\in \GL_n(\CC)$, and $Q\in M_n$.
Since matrix multiplication $g\mapsto Qg$ is linear in the entries of $g$ and $f$ is homogeneous of degree $d$, it follows that $\Ind^d(\rho, V):=(\Ind^d \rho, \Ind^d V)$ is an object of $\Rep^d \GL_n(\CC)$.
Given $\phi\in\Hom_{S_n}(V,W)$, define $\Ind^d \phi:\Ind^d V\to \Ind^d W$ by $(\Ind^d \phi f)(Q) = \phi(f(Q))$ for all $Q\in M_n$.
Then $\Ind^d:\Rep S_n\to \Rep^d \GL_n(\CC)$ is a functor.
\begin{example}
  Let $1_n$ and $\sgn_n$ denote the trivial and sign representations of $S_n$.
  From \eqref{eq:indv},
  \begin{align}
    \label{eq:triv}
    \Ind^d 1_n & = \{f\in P^d(M_n)\mid f(wQ)=f(Q) \text{ for all $w\in S_n$, $Q\in M_n$}\}\\
    \label{eq:sgn}
    \Ind^d \sgn_n & = \{f\in P^d(M_n)\mid f(wQ)=\sgn(w)f(Q) \text{ for all $w\in S_n$, $Q\in M_n$}\}.
  \end{align}
\end{example}
\subsection{Frobenius Reciprocity}
\label{sec:frob-recipr}
\begin{theorem}[Frobenius reciprocity]
  \label{theorem:adj}
  For each $d\geq 0$, the functor $\Ind^d:\Rep S_n\to \Rep^d \GL_n(\CC)$ is right adjoint to the functor $\Res^d$.
  In other words, for every object $(\sigma, U)\in \Rep^d \GL_n(\CC)$ and every object $(\rho,V)$ of $\Rep S_n$, there is a natural isomorphism
  \begin{equation}
    \label{eq:adj}
    \Hom_{\GL_n(\CC)}(U,\Ind^d V) \cong \Hom_{S_n}(\Res^d U, V).
  \end{equation}
\end{theorem}
\begin{proof}
  For each $f\in \Hom_{\GL_n(\CC)}(U,\Ind^d V)$ define $\Phi(f):U\to V$ by
  \begin{displaymath}
    \Phi(f)(u)=f(u)(1),
  \end{displaymath}
  where $1$ denotes the identity matrix in $\GL_n(\CC)$.
  Then
  \begin{align*}
    \Phi(f)(\sigma(w)u) & =f(\sigma(w)u)(1) & \text{[defn. of $\Phi$]} \\
                        & =\Ind^d\rho(w)f(u)(1) & \text{[$f$ is a homomorphism]}\\
                        & =f(u)(w) & \text{[defn. of $\Ind^d\rho$]}\\
                        & =\rho(w)f(u)(1) & \text{[since $f\in \Ind^d V$]}\\
                        & =\rho(w)\Phi(f)(u),
  \end{align*}
  so that $\Phi(f)\in \Hom_{S_n}(\Res^d U, V)$.
  Conversely, given $h\in\Hom_{S_n}(\Res^d U, V)$, define $\Psi(h):U\to P^d(M_n)\otimes V$ by
  \begin{displaymath}
    \Psi(h)(u)(x) = h(\sigma(x)u) \text{ for all $x\in M_n$}.
  \end{displaymath}
  Note that, a priori, $\sigma(x)$ is defined only for $x\in GL_n(\CC)$.
  However, since $\sigma$ is polynomial in the entries of $x$, it extends uniquely to $M_n$.
  Then
  \begin{align*}
    \Psi(h)(u)(wx) & = h(\sigma(wx)u) & \text{[defn. of $\Psi$]}\\
                   & = \rho(w)h(\sigma(x)u) & \text{[$h$ is a homomorphism]}\\
                   & = \rho(w)\Psi(h)(u)(x), & \text{[defn. of $\Psi$]}
  \end{align*}
  so that $\Psi(h)(u)\in \Ind^d V$.
  Moreover,
  \begin{align*}
    \Psi(h)(\sigma(g)u)(x) & = h(\sigma(x)\sigma(g)u) & \text{[defn. of $\Psi$]}\\
                           & = h(\sigma(xg)u) & \text{[$\sigma$ is a representation]}\\
                           & = \Psi(h)(u)(xg) & \text{[defn. of $\Psi$]}\\
                           & = \Ind^d\rho(g)\Psi(h)(u)(x). & \text{[defn. of $\Ind^d \rho$]}
  \end{align*}
  Thus $\Psi$ defines a homomorphism
  \begin{displaymath}
    \Hom_{S_n}(\Res^d U, V)\to \Hom_{\GL_n(\CC)}(U,\Ind^d V).  
  \end{displaymath}
  The mutual inverses $\Phi$ and $\Psi$ give the desired natural isomorphism.
\end{proof}
\subsection{Induction of Permutation Representations}
\label{sec:induct-perm-repr}
Let $X$ be a set with an $S_n$-action.
Let $\CC[X]$ denote the space of all complex-valued functions of $X$.
Then $\CC[X]$ affords a representation $\sigma_X:S_n\to GL(\CC[X])$ via:
\begin{displaymath}
  (\sigma_X(w)f)(x) = f(w^{-1}x).
\end{displaymath}
Representations of the form $\sigma_X$ are known as permutation representations of $S_n$.
The space $P^d(M_n)\otimes \CC[X]$ may be regarded as the space of functions $f:M_n\times X\to \CC$ such that $f(Q,x)$ is a homogeneous polynomial of degree $d$ in the entries of $Q$ for every $x\in X$.
In terms of this identification,
\begin{displaymath}
  \Ind^d \CC[X] = \{f\in P^d(M_n)\otimes \CC[X]\mid f(wQ,wx)=f(Q,x)\text{ for $w\in S_n$, $Q\in M_n$}\}.
\end{displaymath}
Suppose that $X$ is a transitive $S_n$-space.
Fix $x_0\in X$.
Then, for any $f\in \Ind^d \CC[X]$, $f(Q,x) = f(w^{-1}Q,x_0)$, where $w$ is an element of $S_n$ such that $x=wx_0$.
It follows that the map $r:\Ind^d\CC[X]\to P^d(M_n)$ defined by $r(f)(Q) = f(Q,x_0)$ is injective.
Its image in $P^d(M_n)$ is the subspace
\begin{displaymath}
  K^d_X = \{f\in P^d(M_n)\mid f(wQ)=f(Q) \text{ for all $w\in S_n$ with $wx_0=x_0$}\}.
\end{displaymath}
The vector space $K^d_X$ affords a representation of $\GL_n(\CC)$ via
\begin{displaymath}
  (\kappa^d_X(g)f)(Q)=f(Qg).
\end{displaymath}
We have proved the following result:
\begin{theorem}
  \label{theorem:perm-ind}
  The map $r:\Ind^d \CC[X]\to K^d_X$ defined by $r(f)(Q)=f(Q,x_0)$ gives rise to an isomorphism of polynomial representations $(\Ind^d\rho_X,\Ind^d\CC[X])\to (\kappa^d_X,K^d_X)$.
\end{theorem}
\section{Characters and Symmetric Functions}
\label{sec:char-symm-funct}
\subsection{Multipartite Partition Functions and Characters}
\label{sec:mult-part-funct}
In this section we compute the characters of some induced representations in terms of the functions $p_k(\xx)$ and $q_k(\xx)$ defined in Section~\ref{sec:plethysm}.
\begin{lemma}
  \label{lemma:ch-ind-triv-sgn}
  For every positive integer $n$,
  \begin{align}
    \label{eq:ch-ind-triv}
    \ch \Ind^d 1_n & = \sum_{\{\xx\in \NN^n : |\xx|=d\}}p_n(\xx)t^\xx\\
    \label{eq:ch-ind-sgn}
    \ch \Ind^d \sgn_n & = \sum_{\{\xx\in \NN^n : |\xx|=d\}}q_n(\xx)t^\xx.
  \end{align}
\end{lemma}
\begin{proof}
  Let $M(d,n)$ denote the set of all matrices with entries in $\NN$ that sum to $d$.
  For $A\in M(d,n)$ of the form $A=(a_{ij})$, let $q^A$ denote the monomial $\prod_{1\leq i,j\leq n} q_{ij}^{a_{ij}}$.
  Then $\{q^A\mid A\in M(d,n)\}$ is a basis of $P^d(M_n)$.
  A general element of $P^d(M_n)$ is of the form:
  \begin{equation}
    \label{eq:f}
    f = \sum_{A\in M(d,n)} f_A q^A, \quad f_A\in \CC.
  \end{equation}

  By \eqref{eq:triv}, an element $f\in \Ind^d 1_n$ if and only if $f_A=f_{wA}$ for each $w\in S_n$.
  Therefore $\Ind^d 1_n$ has a basis indexed by $S_n$-orbits in $M(d,n)$, where $S_n$ acts by permutation of rows.
  The basis element corresponding to the orbit of $A$ is an eigenvector for $\Ind^d 1_n(\diag(t_1,\dotsc,t_n))$ with eigenvalue $t^\xx$, where $\xx$ is the sum of the rows of $A$.
  Thus the basis elements of $\Ind^d 1_n$ that contribute to the monomial $t^\xx$ in $\ch \Ind^d 1_n$ are in bijection with vector partitions of $\xx$ with at most $n$ parts, giving \eqref{eq:ch-ind-triv}.

  By \eqref{eq:sgn}, an element $f$ of the from \eqref{eq:f} lies in $\Ind^d\sgn_n$ if and only if $f_A=\sgn(w)f_{wA}$ for all $w\in S_n$.
  Suppose that two rows of $A$ are equal.
  Let $w$ be the transposition in $S_n$ that interchanges these two rows.
  Then $wA=A$ and $\sgn(w)=-1$.
  We have $f_A=f_{wA}=-f_A$, so $f_A=0$.
  Therefore $\Ind^d \sgn_n$ has a basis indexed by $S_n$-orbits in $M(d,n)$ consisting of matrices with distinct rows, and \eqref{eq:ch-ind-sgn} follows.
\end{proof}
By padding with $0$'s if necessary, write $\lambda\in \Lambda(d,n)$ as $(\lambda_1,\dotsc,\lambda_n)$.
Let $\delta = (n-1,n-2,\dotsc,1,0)$.
For $w\in S_n$, let $w\cdot \delta = (n-w(1),\dotsc,n-w(n))$.
\begin{corollary}
  \label{cor:trvi-sign-mult}
  For every partition $\lambda\in \Lambda(d,n)$, the multiplicities of the trivial and sign representations of $S_n$ in $\Res^d W_\lambda$ are given by:
  \begin{align*}
    r_{\lambda,(n)} & = \sum_{w\in S_n} \sgn(w)p_n(\lambda+\delta-w\cdot \delta)\\
    r_{\lambda,(1^n)} & = \sum_{w\in S_n} \sgn(w) q_n(\lambda+\delta-w\cdot \delta).
  \end{align*}
  In the above expressions, for $\xx\in \mathbf Z^n-\NN^n$, it should be understood that $p_n(\xx)=0$.
\end{corollary}
The first identity is \cite[Theorem~4.5]{ASSD}.
\begin{proof}
  By Frobenius reciprocity (Theorem~\ref{theorem:adj}),
  \begin{displaymath}
    r_{\lambda\mu} = \langle s_\lambda, \ch \Ind^d V_\mu\rangle, 
  \end{displaymath}
  where $\langle\cdot,\cdot\rangle$ denotes the Hall inner product on symmetric polynomials.
  A well-known formula of Frobenius \cite[Theorem 5.4.10]{rtcv} says that this is the coefficient of $t^{\lambda+\delta}$ in $a_\delta\ch \Ind^d V_\mu$, where $a_\delta$ is the Vandermonde determinant $\det (t_i^{n-j})_{n\times n}=\sum_{w\in S_n} \sgn(w)t^{w\cdot \delta}$.
  By Lemma~\ref{lemma:ch-ind-triv-sgn},
  \begin{align*}
    r_{\lambda,(n)} & = \text{ coefficient of $t^{\lambda+\delta}$ in } \sum_{w\in S_n}\sgn(w)\sum_{|\xx|=d} p_n(\xx)t^{\xx+w\cdot \delta},\\
    r_{\lambda,(1^n)} & = \text{ coefficient of $t^{\lambda+\delta}$ in } \sum_{w\in S_n}\sgn(w)\sum_{|\xx|=d} q_n(\xx)t^{\xx+w\cdot \delta}.
  \end{align*}
  Since $\xx+w\cdot\delta=\lambda+\delta$ if and only if $\xx=\lambda+\delta-w\cdot\delta$, the formulae in Corollary~\ref{cor:trvi-sign-mult} follow.
\end{proof}
\begin{example}
  \label{example:2-part}
  Take $\lambda=(\lambda_1,\lambda_2)$ and any $n\geq 2$.
  Then we pad $\lambda$ with $n-2$ zeros and write it as $(\lambda_1,\lambda_2,0,\dotsc,0)$.
  We get
  \begin{displaymath}
    r_{\lambda,(n)} = \sum_{w\in S_n} \sgn(w)p_n(\lambda_1-1+w(1),\lambda_2-2+w(2),-3+w(3),\dotsc,-n+w(n)).
  \end{displaymath}
  If $-i+w(i)\geq 0$ for all $i> 2$, then we must have $w(i)=i$ for all $i>2$, and so the sum above reduces to
  \begin{displaymath}
    p_n(\lambda_1,\lambda_2,0,\dotsc,0)-p_n(\lambda_1+1,\lambda_2-1,0,\dotsc,0) = p_n(\lambda_1,\lambda_2)-p_n(\lambda_1+1,\lambda_2-1).
  \end{displaymath}
  Similarly,
  \begin{displaymath}
    r_{\lambda,(1^n)} = q_n(\lambda_1,\lambda_2)-q_n(\lambda_1+1,\lambda_2-1).
  \end{displaymath}
\end{example}
\begin{corollary}
  [Unimodality of bipartite partitions]
  For any integers $x_1\geq x_2\geq 1$, and any $n\geq 0$, we have:
  \begin{align*}
    p_n(x_1,x_2) & \geq p_n(x_1+1,x_2-1),\\
    q_n(x_1,x_2) & \geq q_n(x_1+1,x_2-1).
  \end{align*}
\end{corollary}
\begin{proof}
  As explained in Example~\ref{example:2-part}, the difference the left-hand side minus the right-hand side of each identity appears as the multiplicity of the trivial or sign representation in $\Res^{x_1+x_2}W_{(x_1,x_2)}$.
\end{proof}
\begin{remark}
  The first inequality is a result of Kim and Hahn \cite{Kim1997} (see \cite[Section~4.2]{ASSD} for a strengthening of this result).
  The second appears to be hitherto unpublished.
\end{remark}
\subsection{Proof of the Plethystic Formula}
\label{sec:proof-plethyst-form}
For a family $(\sigma,W)=\{(\sigma_d,W_d)\}_{d\geq 0}$, where $(\sigma_d,W_d)$ is an object of $\Rep^d GL_n(\CC)$, define its character to be the symmetric formal power series in the the variables $t_1,\dotsc,t_n$
\begin{displaymath}
  \ch \sigma(t_1,\dotsc,t_n)=\sum_{d\geq 0} \ch \sigma_d(t_1,\dotsc,t_n).
\end{displaymath}

For a representation $(\rho, V)$ of $S_n$, let $(\Ind \rho,\Ind V)$ denote the family \linebreak $\{(\Ind^d \rho, \Ind^d V)\}_{d\geq 0}$.
Then Lemma~\ref{lemma:ch-ind-triv-sgn} can be rewritten as:
\begin{align*}
  \ch \Ind 1_n &= \sum_{\xx\in \NN^n} p_n(\xx)t^\xx,\\
  \ch \Ind \sgn_n &= \sum_{\xx\in \NN^n} q_n(\xx)t^\xx.
\end{align*}

For each partition $\mu=(\mu_1,\dotsc,\mu_m)$ of $n$ let
\begin{displaymath}
  X_\mu = \textstyle{\{(S_1,\dotsc,S_m)\mid S_1\coprod \dotsb \coprod S_m = [n]\}},
\end{displaymath}
the collection of all ordered set partitions of $[n]$ into $m$ parts, of sizes $\mu_1,\dotsc,\mu_m$.
Then $X_\mu$ inherits an $S_n$-action from $[n]$.
\begin{lemma}
  \label{lemma:char-perm-ind}
  For every partition $\mu=(\mu_1,\dotsc,\mu_m)$ of $n$,
  \begin{displaymath}
    \ch \Ind \CC[X_\mu] = \prod_{i=1}^m \sum_{\xx\in \NN^n}p_{\mu_i}(\xx)t^\xx.
  \end{displaymath}
\end{lemma}
\begin{proof}
  Let $x_0=(\{1,\dotsc,\mu_1\},\{\mu_1+1,\dotsc,\mu_1+\mu_2\},\dotsc,\{\mu_1+\dotsb+\mu_{m-1}+1,\dotsc,n\})$, an element of $X_\mu$.
  The stabilizer of $x_0$ in $S_n$ is the Young subgroup $S_\mu:=S_{\mu_1}\times \dotsb S_{\mu_m}$ of $S_n$.
  By Theorem~\ref{theorem:perm-ind},
  \begin{displaymath}
    \Ind^d \CC[X_\mu] = \{f\in P^d(M_n)\mid f(wQ)=f(Q) \text{ for all } w\in S_\mu,\; Q\in M_n\}.
  \end{displaymath}
  Proceeding as in the proof of Theorem~\eqref{eq:ch-ind-triv}, we see that a basis of $\Ind^d \CC[X_\mu]$ is indexed by $S_\mu$-orbits of matrices in $M(d,n)$, where $S_\mu$ acts by permuting blocks of rows of sizes $\mu_1,\mu_2,\dotsc,\mu_m$.
  The $i$th block of rows may be regarded as a vector partition with at most $\mu_i$-parts of some vector $\xx_i$.
  It contributes to $t^\xx$ in the total character when $\xx_1+\dotsb+\xx_m=\xx$.
\end{proof}
\begin{theorem}
  \label{theorem:indchar}
  For every representation $V$ of $S_n$,
  \begin{displaymath}
    \ch \Ind V = \mathcal F(V)[H].
  \end{displaymath}
\end{theorem}
\begin{proof}
  Since $\Ind^d$ is an additive functor and $f\mapsto f[H]$ is also additive, it suffices to prove the identity in the theorem for any family of representations whose characters span the space of class functions on $S_n$.
  In particular, it suffices to prove the identity for the family $\CC[X_\mu]$ of permutation representations, as $\mu$ runs over all partitions of $n$.
  We have
  \begin{align*}
    \ch \Ind \CC[X_\mu] & = \prod_i \sum_{\xx\in \NN^d} p_{\mu_i}(\xx) & [\text{Lemma~\ref{lemma:char-perm-ind}}\\
                        & = \prod_i h_{\mu_i}[H] & [\text{Theorem~\ref{theorem:ehH}}]\\
                        & = h_\mu[H]\\
                        & = \mathcal F(\CC[X_\mu])[H],
  \end{align*}
  as required.
\end{proof}
We are now in a position to give a representation-theoretic proof of Littlewood's plethystic formula for restriction coefficients (Theorem~\ref{lit}):
\begin{align*}
  \langle \mathcal F(\Res^d W_\lambda), s_\mu\rangle & = \dim \Hom_{S_n}(\Res^d W_\lambda, V_\mu)\\
                                                     & = \dim \Hom_{\GL(n)}(W_\lambda,\Ind^d V_\mu) & \text{[Theorem~\ref{theorem:adj}]}\\
                                                     & = \langle s_\lambda, \ch \Ind^d V_\mu\rangle\\
                                                     & = \langle s_\lambda, s_\mu[H]\rangle, & \text{[Theorem~\ref{theorem:indchar}]}
\end{align*}
as required.


\begin{thebibliography}{10}

\bibitem{Peal}
  P.~Alexandersson.
  \newblock Schur positivity.
  \newblock {\em Symmetric functions catalog}, accessed 8th April 2020.
  \newblock \url{https://www.math.upenn.edu/~peal/polynomials/schurPositivity.htm}
  
\bibitem{assaf-speyer}
  S.~H. Assaf and D.~E. Speyer.
  \newblock Specht modules decompose as alternating sums of restrictions of
  {S}chur modules.
  \newblock {\em Proc. Amer. Math. Soc.}, 148(3):1015--1029, 2020.
  \newblock \url{https://doi.org/10.1090/proc/14815}  

\bibitem{CRY}
  J.~O. Carbonara, J.~B. Remmel, and M.~Yang.
  \newblock Exact formulas for the plethysm $s_2[s_{(1^a,b)}]$ and
  $s_{1^2}[s_{(1^a,b)}]$.
  \newblock Technical report, MSI, 1992.

\bibitem{CRY2}
  J.~O. Carbonara, J.~B. Remmel, and M.~Yang.
  \newblock A combinatorial rule for the Schur function expansion of the plethysm $s_{(1^a,b)}[p_k]$.
  \newblock {\em Linear Multilinear Algebra}, 39:4, 341--373, 1995.
  \newblock \url{http://dx.doi.org/10.1080/03081089508818407}

\bibitem{MR3685118}
  L.~Carini.
  \newblock On the multiplicity-free plethysms {$p_2[s_\lambda]$}.
  \newblock {\em Ann. Comb.}, 21(3):339--352, 2017.
  \newblock \url{https://doi.org/10.1007/s00026-017-0354-0}

\bibitem{MR1661367}
  L.~Carini and J.~B. Remmel.
  \newblock Formulas for the expansion of the plethysms {$s_2[s_{(a,b)}]$} and
  {$s_2[s_{(n^k)}]$}.
  \newblock {\em Discrete Math.}, 193(1-3):147--177, 1998.
  \newblock Selected papers in honor of Adriano Garsia (Taormina, 1994).
  \newblock \url{https://doi.org/10.1016/S0012-365X(98)00139-3}

\bibitem{MR1331743}
  C.~Carr\'{e} and B.~Leclerc.
  \newblock Splitting the square of a {S}chur function into its symmetric and
  antisymmetric parts.
  \newblock {\em J. Algebraic Combin.}, 4(3):201--231, 1995.
  \newblock \url{https://doi.org/10.1023/A:1022475927626}

\bibitem{MR777698}
  Y.~M. Chen, A.~M. Garsia, and J.~Remmel.
  \newblock Algorithms for plethysm.
  \newblock In {\em Combinatorics and algebra ({B}oulder, {C}olo., 1983)},
  volume~34 of {\em Contemp. Math.}, pages 109--153. Amer. Math. Soc.,
  Providence, RI, 1984.

\bibitem{frob-ind}
  F.~G.~Frobenius.
  \newblock \"Uber Relationen zwischien den Characteren einer Gruppe und denen
  ihrer Untergruppen.
  \newblock {\em S'ber. Akad. Wiss. Berlin}, 501--515, 1898.
  \newblock \url{https://doi.org/10.3931/e-rara-18872}
  
\bibitem{frob}
  F.~G.~Frobenius.
  \newblock \"Uber die Charaktere der symmetrischen Gruppe.
  \newblock {\em S'ber Akad. Wiss. Berlin}, 148--166, 1900.
  \newblock \url{http://doi.org/10.3931/e-rara-18862}
  
\bibitem{Nate}
  N.~{Harman}.
  \newblock {Representations of monomial matrices and restriction from $GL_n$ to
    $S_n$}, 2018.
  \newblock \url{https://www.arXiv.org/abs/1804.04702}

\bibitem{Kim1997}
  J.~K. Kim and S.~G. Hahn.
  \newblock Partitions of bipartite numbers.
  \newblock {\em Graphs and Combinatorics}, 13(1):73--78, 1997.
  \newblock \url{https://doi.org/10.1007/BF01202238}

\bibitem{MR2035305}
  T.~M. Langley and J.~B. Remmel.
  \newblock The plethysm {$s_\lambda[s_\mu]$} at hook and near-hook shapes.
  \newblock {\em Electron. J. Combin.}, 11(1):Research Paper 11, 26, 2004.
  \newblock \url{http://www.combinatorics.org/Volume_11/Abstracts/v11i1r11.html}

\bibitem{MR0002127}
  D.~E. Littlewood.
  \newblock {\em The {T}heory of {G}roup {C}haracters and {M}atrix
    {R}epresentations of {G}roups}.
  \newblock Oxford University Press, New York, 1940.

\bibitem{MR95209}
  D.~E. Littlewood.
  \newblock Products and plethysms of characters with orthogonal, symplectic and
  symmetric groups.
  \newblock {\em Canadian J. Math.}, 10:17--32, 1958.
  \newblock \url{https://doi.org/10.4153/CJM-1958-002-7}

\bibitem{MR2765321}
  N.~A. Loehr and J.~B. Remmel.
  \newblock A computational and combinatorial expos\'{e} of plethystic calculus.
  \newblock {\em J. Algebraic Combin.}, 33(2):163--198, 2011.
  \newblock \url{https://doi.org/10.1007/s10801-010-0238-4}

\bibitem{Mackey}
  G.~W.~Mackey.
  \newblock On induced representations of groups.
  \newblock {\em Amer. J. Math.}, 73(3):576--592, 1951.
  \newblock \url{https://doi.org/10.2307/2372309}
  
\bibitem{ASSD}
  S.~Narayanan, D.~{Paul}, A.~{Prasad}, and S.~{Srivastava}.
  \newblock Character polynomials and the restriction problem, 2020.
  \newblock \url{https://www.arxiv.org/abs/2001.04112}

\bibitem{NPS}
  S.~P. Narayanan, D.~Paul, and S.~Srivastava.
  \newblock The multiset partition algebra, 2019.
  \newblock \url{https://www.arxiv.org/abs/1903.10809}

\bibitem{2016arXiv160506672O}
  R.~{Orellana} and M.~{Zabrocki}.
  \newblock {Symmetric group characters as symmetric functions}, 2018.
  \newblock \url{https://www.arxiv.org/abs/1605.06672v4}

\bibitem{dpthesis}
  D.~Paul.
  \newblock {\em The Multiset partition algebra}.
  \newblock PhD thesis, Homi Bhabha National Institute, submitted 2020.

\bibitem{rtcv}
  A.~Prasad.
  \newblock {\em Representation Theory: A Combinatorial Viewpoint}.
  \newblock Cambridge University Press, New Delhi, 2015.

\bibitem{MR1272068}
  T.~Scharf and J.-Y. Thibon.
  \newblock A {H}opf-algebra approach to inner plethysm.
  \newblock {\em Adv. Math.}, 104(1):30--58, 1994.
  \newblock \url{https://doi.org/10.1006/aima.1994.1019}

\bibitem{schurthesis}
  I.~Schur.
  \newblock {\em Ueber eine Klasse von Matrizen, die sich einer gegebenen Matrix zuordnen lassen}.
  \newblock Dissertation, Berlin 1901.

\bibitem{MR1676282}
  R.~P. Stanley.
  \newblock {\em Enumerative combinatorics. {V}ol. 2}, volume~62 of {\em
    Cambridge Studies in Advanced Mathematics}.
  \newblock Cambridge University Press, Cambridge, 1999.
\end{thebibliography}
\end{document}